\documentclass[12pt, a4paper]{amsart}
\usepackage{amscd,amssymb,eucal,eufrak,mathrsfs,amsmath}
\usepackage{hyperref} 
\usepackage{graphics,colortbl}

\input xy
\xyoption{all}
\usepackage{epsfig}

\newtheorem{thm}{Theorem}[section]
\newtheorem{cor}[thm]{Corollary}
\newtheorem{lem}[thm]{Lemma}
\newtheorem{prop}[thm]{Proposition}

\newtheorem{question}[thm]{Question}                                     
\theoremstyle{definition}
\newtheorem{defin}[thm]{Definition}

\theoremstyle{remark}
\newtheorem{remark}[thm]{Remark}
\newtheorem{remarks}[thm]{Remarks}

\newtheorem{examples}[thm]{Examples}
\numberwithin{equation}{section}

\newcommand{\delete}[1]{} % Comment out text.
\newcommand{\nt}{\noindent}

\def\om{\omega}

\def\a{\alpha}

\newcommand{\Ga}{\Gamma}

\newcommand{\del}{\delta}

\newcommand{\la}{\lambda}
\newcommand{\La}{\Lambda}

\newcommand{\sk}{\vskip 0.2cm}

\newcommand{\nl}{\newline}
\newcommand{\ben}{\begin{enumerate}}

\newcommand{\een}{\end{enumerate}}
\newcommand{\bit}{\begin{itemize}}

\newcommand{\eit}{\end{itemize}}

\def\R {{\mathbb R}}
\def\N {{\mathbb N}}
\def\Z {{\mathbb Z}}

\def\F{{\mathcal F}}

\def\Aut{{\mathrm Aut}\,}

\def\Aut{{\rm{Aut}}}
\newcommand{\lan}{\langle}
\newcommand{\ran}{\rangle}

\newcommand{\nor}{\vartriangleleft}

\newcommand{\rest}{\upharpoonright}

\newcommand{\br}{\vspace{3 mm}}

\newcommand{\ch}{\mathbf{1}}

\newcommand{\bA}{\mathbf{A}}
\newcommand{\bM}{\mathbf{M}}

\def\QED{\nobreak\quad\ifmmode\roman{Q.E.D.}\else{\rm Q.E.D.}\fi}

\begin{document}

\title[]{Group actions on treelike compact spaces}

%Authors
\author[]{Eli Glasner}
\address{Department of Mathematics,
	Tel-Aviv University, Ramat Aviv, Israel}
\email{glasner@math.tau.ac.il}
\urladdr{http://www.math.tau.ac.il/$^\sim$glasner}

\author[]{Michael Megrelishvili}
\address{Department of Mathematics,
	Bar-Ilan University, 52900 Ramat-Gan, Israel}
\email{megereli@math.biu.ac.il}
\urladdr{http://www.math.biu.ac.il/$^\sim$megereli}

%1502
\date{May 27, 2019}

%2808   
\thanks{This research was supported by a grant of the Israel Science Foundation (ISF 668/13)}



\begin{abstract}  
	We show that group actions on many treelike compact spaces are not too complicated dynamically. 
We first observe that an old argument of Seidler \cite{Seid-90}, implies that every action of a topological group $G$ 
on a regular continuum is null and therefore also tame. As every local dendron is regular, one concludes that every action of $G$ on a local dendron is null. We then use a more direct method to show that every continuous group action of $G$ on a dendron is Rosenthal representable, hence also tame. Similar results are obtained for median pretrees. 
As a related result we show that Helly's selection principle can be extended to bounded monotone sequences 
defined on median pretrees (e.g., dendrons or linearly ordered sets).  
 Finally, we point out some applications of these results to continuous group  actions on dendrites. 
\end{abstract}

\maketitle

\setcounter{tocdepth}{1}
 \tableofcontents

%3009 ADDING IN THIS SECTION TITLE "and preliminaries"
\section{Introduction and preliminaries} 
 
 %2409
%  I think we need to briefly define here the notions of  tameness, 
% Rosenthal representability, and WRN 

%3009 
One of the 
%starting points
motivations for writing the present work was to show 
that group actions on dendrons are Rosenthal representable, hence tame (see Definition \ref{d:repr} below).  
Representations on Banach spaces with ``good" geometry lead to a 
natural hierarchy in the world of continuous actions $G \curvearrowright X$ 
of topological groups $G$ on topological spaces $X$.  
In particular, representations on Banach spaces without a copy of $l_1$ (we call them {\it Rosenthal} Banach spaces) play a very important role in this hierarchy. 
%Among the reasons why such representations are important we mention: 
According to the Rosenthal $l_1$-dichotomy \cite{Ros0}, 
and the corresponding Dynamical Bourgin-Fremlin-Talagrand dichotomy \cite{GM-HNS,GM-survey},
there is a sharp dichotomy for 
%1502 adding "metrizable" 
metrizable 
dynamical systems;
either their enveloping semigroup is of cardinality smaller or equal to that of the continuum, or it
is very large and contains a copy of the Stone-\v{C}ech compactification of $\N$ (the set of natural numbers). 

When $X$ is compact 
%1502 adding "metrizable"  
metrizable, 
in the first case, such a dynamical system $(G,X)$ is called {\it tame}. 
By A. K\"{o}hler's \cite{Koh} definition, tameness  
 means that for every continuous real valued function 
$f: X \to \R$ the family of functions $fG:=\{fg\}_{g \in G}$ is ``combinatorially small"; 
namely, $fG$ does not contain an {\it independent sequence}.

Recall the classical definition from \cite{Ros0}: a sequence $f_n$ of real valued functions on a set $X$ 
is said to be \emph{independent} if
there exist real numbers $a < b$ such that
$$
\bigcap_{n \in P} f_n^{-1}(-\infty,a) \cap  \bigcap_{n \in M} f_n^{-1}(b,\infty) \neq \emptyset
$$
for all finite disjoint subsets $P, M$ of $\N$.
%%%%
As in \cite{GM-tame,GM-tLN} we say that a bounded family of real valued functions is a 
{\it tame family} if it does not contain an independent sequence.   
 For 
 %1502 adding "compact" (referee's remark) 
 compact 
 metrizable systems $X$ the two concepts discussed above coincide: 
 $(G,X)$ is tame iff it is Rosenthal representable.  
 The case of metrizable $X$ admits 
 also interesting enveloping semigroup characterization: $(G,X)$ is tame 
% iff its enveloping semigroup $E(X)$ has a cardinality at most that of the  reals, 
 iff every element
 %1502  replacing E(X) by E(G,X)$ (referee's suggestion) 
  $p \in E(G,X)$ is a Baire class 1 function $p: X \to X$ (see  \cite{GMU,GM-survey}). 
 %1502 including here definition of the enveloping semigroup 
 Recall that the {\it enveloping semigroup} $E(G,X)$ of a compact $G$-system $X$ is the pointwise closure of the set of all $g$-translations $X \to X$ ($g \in G$) in $X^X$. 

We next recall the definition of a Banach representation of a dynamical system. 
For a Banach space $V$ denote by $Is(V)$ the topological group (equipped with the strong operator topology) of all linear isometries $V \to V$.  The group $Is(V)$ acts on the dual space $V^*$ and its weak-star compact subsets. 

 \begin{defin} \label{d:repr} (See \cite{Me-nz,GM-HNS,GM-survey}) 
%	 Let $G \curvearrowright X$ be a continuous action of a topological group $G$ on a topological space $X$. 
	% and $S \times X \to X$ be an action. 
	A \emph{representation} of an action  $G \curvearrowright X$ on a Banach space $V$ is a pair $(h,\a)$, where 
	$h: G \to Is(V)$ is a strongly continuous co-homomorphism 
	%into the topological group of all linear isometries $V \to V$ 
	%with strong operator topology 
	and $\a: X \to V^*$ is a weak-star continuous bounded map 
	%1502  correction of the referees 
	%   $\a: K \to V^*$ 
	into the dual of $V$ such that $(h,\a)$ respect the  original action and the induced dual action. That is, 
	$$
	\lan v,\a(gx) \ran = \lan h(g)(v),\a(x) \ran   \ \ \ \text{for all} \  v \in V, g \in G, x \in X.  
	$$ 
\end{defin}

We say that $(G,X)$ is {\it Rosenthal representable}, or WRN (Weakly Radon-Nikodym),  if there exist: a Rosenthal Banach space $V$ and a representation $(h,\a)$ as above such that $\a$ is a topological embedding. 
 For compact $G$-systems $X$ this concept was defined in \cite{GM-rose}.  
%It would be interesting to find sufficient conditions for Rosenthal representability of (not necessarily compact) topological spaces. 
  The class of all Rosenthal representable dynamical systems and its subclass of tame systems are quite large. 
  When considering the trivial action of $G$ on a space $X$, this definition reduces to the purely topological notion
  of a WRN space.
  
  For motivation, properties and examples see, for example,  
  \cite{GM-rose,GM-survey,GM-tLN} and the references thereof.  
A relevant recent result is that, for
 every circularly  (in particular, linearly) ordered compact space $K$, the action $H_+(K) \curvearrowright K$, of the topological group $H_+(K)$ of all order-preserving homeomorphisms of $K$ on $K$, 
 is Rosenthal representable, \cite{GM-c}. 
 
 %3009 A עד כאן השינויים במבוא   ?????
 
\br

 Recall that a \textit{continuum} is a compact Hausdorff connected space. 
 A continuum $D$ is said to be a \textit{dendron} \cite{Mill-Wattel} if every pair of distinct points 
 $u,v$ can be separated in $D$ by a third point $w$.
%eli  
A metrizable dendron is called a \textit{dendrite}. 
The class of dendrons is an important class of 1-dimensional 
%1502 tree like 
treelike
compact spaces. 
%\footnote{Certain Julia sets are dendrites.}  
A  compact space is said to be a \textit{local dendron} if each of its points admits 
a closed neighborhood which is a dendron. Similarly, a compact space is called a 
\textit{local dendrite} 
if each of its points admits a closed neighborhood which is a dendrite. Clearly, the circle is a local dendrite.

%2409
In a dendron $X$, for every pair of points $u,v$ in $X$  one defines a ``generalized arc" 
	$$
	[u,v]=\{x \in X: \ x \ \text{separates $u$ from $v$}\} \cup \{u,v\}. 
	$$
In dendrites the 
%2808  set $[u,v]$ is an {\em arc}, 
generalized arc $[u,v]$ is a {\em real arc}, 
i.e. a topological copy of an interval in the real line. 
This is not necessarily the case for dendrons. Indeed, note that any connected linearly ordered compact space, 
in its interval topology,  is an example of a dendron.

%0306 
	A topological space $X$ is called {\em regular} if every point has a local base for its topology,
	each member of which has finite boundary. 
	For compact Hausdorff spaces this is equivalent to saying that each open cover admits a finer (finite) open cover each member of which has finite  boundary (see e.g. \cite{Kur}). 
	Some works refer to such compact spaces as being  {\em rim-finite}  (see e.g. \cite{Ward}). 
Every (local) dendron is regular \cite[Theorem 21]{Ward}. 	
For more information on dendrons and dendrites
see, for example, \cite{Charatoniks, Mill-Wattel,DM-16}. 
 
 \sk 

Our goal in the first part of this paper is to prove the following two results 
% for dendrons 
(see Theorems \ref{t:nullness} and \ref{t:Dendron} for the proofs).

 \begin{thm} \label{t:nullnessINTRO} 
 	Any action of a group $G$ on a regular %rim-finite 
  %0210	metrizable 
 	continuum is null, hence also tame.
 \end{thm}

\begin{thm} \label{thm-tame}
Let $D$ be a dendron. For every topological group $G$ and continuous action 
$G \curvearrowright D$, the dynamical $G$-system $D$ is Rosenthal representable, hence also tame. 
\end{thm}

In the proof of Theorem \ref{t:nullness} we use a method of Seidler \cite{Seid-90}. 
Note that by a result of Kerr and Li \cite{KL} every null dynamical $G$-system is tame. 
This explains the tameness conclusion in Theorem \ref{t:nullnessINTRO}.

\sk

%0306 
In the proof of Theorem \ref{t:Dendron} 
we use the following useful characterization. 

\begin{thm} \label{t:criter1} \cite[Theorem 6.10]{GM-rose} 
A dynamical $G$-system $X$ is Rosenthal representable if and only if there exists a 
$G$-invariant point-separating bounded family $F$ of continuous functions $X \to \R$ which is 
tame as a family of functions.  
%3009 this was explicitly defined now in introduction (i.e., does not contain an independent sequence).  
\end{thm}

We consider monotone (not necessarily, continuous) functions (Definition \ref{d:monot2}) on dendrons $D$. 
In Theorem \ref{t:MonotFragm} we show that any such function is fragmented (Baire class 1, on dendrites).  
In order to apply Theorem \ref{t:criter1} for dendrons $D$,  
in the role of the family $F$ we consider the set of all continuous monotone functions $D \to [0,1]$. 

%0706 
Later in Theorem \ref{t:pretree} we generalize Theorem \ref{thm-tame} to \textit{compact median pretrees} 
with monotone group actions. This approach
% implies 
 provides 
the following corollary (see Corollary \ref{c:Ends}).

\begin{cor}
Let $X$ be a $\Z$-tree. Denote by $Ends(X)$ the set of all its ends. Then for every monotone group action $G \curvearrowright X$ by homeomorphisms, the induced action of $G$ on the compact space 
$\widehat{X}:=X \cup Ends(X)$ 
is Rosenthal representable. 
\end{cor}

%2309
As a related result we show, in Theorem \ref{GenHellyThm}, that 
Helly's selection principle can be extended to bounded monotone 
sequences of real valued functions defined on median pretrees (e.g., dendrons or linearly ordered sets). 
%3009 
Also in Theorem \ref{t:FragmPretr} we show that 
every monotone real valued function $f: X \to \R$ on a compact Hausdorff pretree is fragmented. 

\sk

In the second part of the paper, 
as applications of Theorems \ref{t:nullnessINTRO} and \ref{thm-tame},
and building on ideas and results from \cite{DM-16} and \cite{MN-16}, 
we easily recover some old results, and prove some new ones,
concerning dynamical systems defined on (local) dendrons.

In Section \ref{cons} we show that 
when a group $G$ acts on a dendrite $X$ with no finite orbits and  
$M \subset X$ is the unique minimal set, then the action on $X$ is strongly proximal
and the action on $M$ is extremely proximal.
In Section \ref{discrete-spectrum} we show that for an amenable group $G$
every infinite minimal set in a dendrite system 
%1502 $(X,G)$ 
$(G,X)$ 
is almost automorphic.
This result was strengthened recently by Shi and Ye \cite{SY-18}, who have shown
that it is actually equicontinuous. In the final section we comment on the special case
where the acting group is the group of integers.

\sk 

%0210   
\nt \textbf{Acknowledment.}  
Thanks are due to Nicolas Monod for enlightening conversations concerning tameness of
actions on dendrites. We also thank him and the organizers of the conference
entitled ``Structure and dynamics of Polish groups", held at the CIB in Lausanne, March 2018,
for the invitation to participate in this conference. The successful conference contributed a great deal
to the progress of this work.
Thanks are due also to Jan van Mill for his helpful advise.

%3009 What about \textbf{Acknowledment}  ?? 
%N. Monod, J. van Mill, organizers of Lausanne conference, ... ?  

\br

\section{Actions of groups on a regular continuum are null}

Following Goodman \cite{Go-74}, 
for a sequence $S = \{s_1, s_2, \dots\} \subset G$ we define the topological sequence entropy of 
a dynamical system $(G, X)$, 
with respect to $S$ and a finite open cover $\bA$ of $X$, by
$$
h_{top}(X,\bA ; S) = \lim_{n \to \infty} n^{-1} \log \left( N \left (\bigvee_{i=1}^{n} s_i(\bA)\right)\right),
$$
where $N(\cdot)$ denotes the minimal cardinality of a subcover. 
We say that $(G, X)$ is {\em null} if $h_{top}(X,\bA ; S) =0$ for all open covers 
$\bA$  of $X$ and all sequences $S$ in $G$.

\sk

The proofs in this section are taken, almost verbatim, from  Seidler's paper
\cite{Seid-90}. For a subset  $A \subset X$ we denote its boundary by $\partial(A)$.

	%\begin{defin}
	%A compact metric space $X$ is called {\em regular} if
	%there exist open covers of $X$ of arbitrarily small mesh such that every 
	%element of the cover has finite boundary. 
	%\end{defin}   

%Recall that a compact space is said to be a \textit{local dendron} if each point admits a closed neighborhood which is a dendron. Similarly one defines \textit{local dendrites} if each point admits a closed neighborhood which is a dendrite.

%The circle is a local dendrite but not a dendrite. 

\begin{lem}\label{l-1}
Let $\bA$ be an open cover of a compact continuum $X$ and let $P$ be the set of 
boundary points of elements of $\bA$. If $\bA$ is a minimal cover then the number of 
elements of $\bA$ is at most the number of elements in $P$.
\end{lem}

\begin{proof}
%2409   several misprints 
 Assume that $\bA$ does not have a proper subcover. 
Then $\bA$ is a finite collection because $X$ is compact. Let $a$ be the number of elements of $\bA$ so
that $\bA = \{A_i: 1 \leq i \leq a \}$. 
Let $A_j \in \bA$, and let $V = \bigcup_{i\not= j} A_i$. 
Note that $V$ is open as a union of open sets. 
Because $\bA$ doesn't have a proper subcover we see that 
$X \setminus V \not= \emptyset$, so that  $\partial(V) \not= \emptyset$, 
since $X$ is connected. 
Further $\partial(V) \subset P$ because $V$ is a finite union of elements of $\bA$. 
As $V$ is open, $\partial(V) \cap V = \emptyset$, so that $\partial(V) \subset A_j$. 
Thus, for each element $A_j$ of $\bA$, 
there must exist a nonempty subset of $P$ contained in $A_j$ 
but disjoint from every other element of $\bA$. 
This implies that the number of elements of $\bA$ be at most the number of elements in $P$.

\end{proof}

\sk
 
\begin{lem}\label{l-2}
Let $G$ be a countable infinite group acting on a regular compact continuum $X$.
Let $S = \{s_0, s_1, s_2, \dots\}$ be a sequence of elements of $G$.
Let $\bA$ be a minimal open cover of $X$ of at least two elements such that every 
element of $\bA$ has finite boundary. Let $L_{\bA}$ be the total number of boundary points 
of elements of $\bA$. 
For each positive integer $n$, 
Let $\bM_n$ be a subcover of
minimum cardinality of $\bigvee_{i=0}^{n-1} s_i(\bA)$
and let $P_n$ be the collection of boundary points 
of elements of $\tilde{\bA} = \bigcup_{i=0}^{n-1} s_i(\bA)$. Then
\begin{enumerate}
\item
 For each positive integer $n$ every boundary point of an element of $\bM_n$ is in $P_n$.
\item
$N(\bigvee_{i=0}^{n-1} s_i(\bA)) \le n L_{\bA}$.
\end{enumerate}
\end{lem}

\sk

\begin{proof}
(1) \  
Let $n$ be a positive integer and let $x$ 
be a boundary point of $M \in \bM_n$. 
By the definition of $\bM_n$, 
for each positive integer $0 \le i <  n$, there exists $V_i \in \tilde{\bA}$ 
such that $M = \bigcap_{i=0}^{n-1} V_i$. 
Let $x$ be a boundary point of $M$ 
and let $B$ be an open set containing $x$; 
clearly $B \cap M \not=\emptyset$ so that $B \cap V_i \not=\emptyset$ for each $i$. 
For each $V_i$, this requires that either $x \in \partial(V_i)$ or $x \in V_i$. Suppose that every
$V_i$ contains $x$. Then $x \in M$ by definition. But this contradicts $x \in \partial(M)$, 
because $M$ is open. Thus $x$ is a boundary point of at least one $V_i$ and thus $x \in P_n$.

\sk

 (2)\ 
Let $n$ be a positive integer. Because each $s_j$ is a homeomorphism the number of boundary 
points of elements of $s_j(\bA)$  is $L_{\bA}$ for every integer $j$. 
This requires that the number of elements of $P_n$ be at most $nL_{\bA}$. 
As every boundary point of an element of $\bM_n$ is in $P_n$ and as $\bM_n$ doesn't 
have a proper subcover, Lemma \ref{l-1} implies that there exist at most $nL_{\bA}$ elements in $\bM_n$. 
The desired result then follows from the definition of $\bM_n$.

\end{proof}

\sk

\begin{thm} \label{t:nullness} 
Every action 
%2808 is this true for every monoid MONOTONE action? 
of a group $G$ on a regular  
%3009 Eli, what is your opinion: "metrizable" can be removed ?
% It seems that indeed the proofs do not require the metrizability assumption 
continuum is null, hence a fortiori tame.
\end{thm}

\begin{proof}
Let $X$ be a regular compact space on which $G$ acts.
Let $\bA$ be a minimal open cover of 
$X$ containing at least two elements such that every element of $\bA$ has finite boundary. 
Let $L_{\bA}$ be the total number of boundary points of elements of $\bA$. 
Given $S = \{s_1, s_2, \dots\} \subset G$,
we have then, from Part (2) of Lemma \ref{l-2}, that
\begin{align*}
h_{top}(X,\bA ; S) & = \lim_{n \to \infty} n^{-1} \log \left( N \left (\bigvee_{i=1}^{n} s_i(\bA)\right)\right)\\
&  \le \lim_{n \to \infty} n^{-1} \log(n L_{\bA}) = 0.
\end{align*}
%2409
%As $\bA$ was an arbitrarily chosen minimal open cover of at least two elements 
%such that every element of $\bA$ has finite boundary, we see that the $S$-entropy of $G$ 
%with respect to every such cover is zero. 
%As $X$ is regular, such covers exist with 
%arbitrarily small mesh, and a refining sequence of such covers must exist.
Thus $h_{top}(X; S)= 0$ and this shows that the system $(G, X)$ is null. 
By a theorem of Kerr and Li \cite{KL} it is also tame.
\end{proof}

%0306 
%0405   moved to Section \ref{discrete-specrun}
%\begin{remark} \label{r:Mobius} \ 
%	\ben 
%	\item 
% In \cite{HWY} the authors show that every tame cascade satisfies the ``Mobius disjointness conjecture". 
% 	\item Moreover, 
%in \cite{AAM} it was proved that every monotone cascade on a local dendrite satisfies the ``Mobius disjointness conjecture". 
%\item 
%Theorem \ref{t:nullness} shows that (2) can be derived from (1) at least for every invertible \footnote{"invertible" can be removed ??} cascade on a local dendrite.  
%\een 
%\end{remark}
%

%2506
\begin{remark}
In \cite[Theorem 12.2]{KL} the authors demonstrate with a simple proof, 
that every action of a convergence group $G$ on a compact space $X$
(in particular, any hyperbolic group acting on its Gromov boundary) is null.
\end{remark}

\br

\section{Dendrons, monotone functions and group actions}

\subsection{Standard betweenness relation and dendrons} 

All the topological spaces in this work are assumed to be Hausdorff.  
	Let $X$ be a connected 
	topological space and $u,v \in X$. 
	As usual, we say that a point $w$ \textit{separates} $u$ and $v$ in $X$ 
	if there exist in $X$ open disjoint neighborhoods $U,V$ of $u$ and $v$ 
	respectively such that $X \setminus\{w\}=U \cup V$. 
%2808 	It is equivalent to the condition that $u$ and $v$ lie in different quasicomponents of the subspace $X \setminus \{w\}$.  

	%%A quasicomponent of $y$ in a space $Y$ is the intersection of all clopen neighborhoods of $y$ in $Y$.  

	For every $u,v$ in 
	$X$ define the ``generalized arc" 
	$$
	[u,v]=\{x \in X: \ x \ \text{separates $u$ from $v$}\} \cup \{u,v\}. 
	$$
		%!!!  (notation of \cite{Mill-Wattel} is $S(u,v)$) 
%2409 
	%It is common to allow the identification $[u,v]=[v,u]$.
	By definition $[u,v]=[v,u]$.

\begin{defin} \label{d:BrelationD} 
	Let $X$ be a topological space and $u,w,v \in X$. 
	We say that $w$ is \textit{between} $u$ and $v$ in $X$ if $w \in [u,v]$. 
	That is, $w$ separates $u$ and $v$ or $w \in \{u,v\}$. 
	This defines a natural betweenness ternary relation on $X$. 
	Denote by $R_B$ this ternary relation. Sometimes we write 
%eee
$\langle u,w,v \rangle$ instead of 
	$(u,w,v) \in R_B$. 
\end{defin}

%By a \textit{continuum} we mean a compact connected space. 
%Recall that a continuum $D$ is said to be a \textit{dendron} \cite{Mill-Wattel} if every pair of distinct points $u,v$ can be separated in $D$ by a third point $w$.
%%eli  
%A metrizable dendron is called a \textit{dendrite}. 
%The class of dendrons  
%is
%% consist 
%an important class of 1-dimensional tree like compact spaces. 
%%\footnote{Certain Julia sets are dendrites.}  
%For more information 
%see for example, \cite{Charatoniks, Mill-Wattel,DM-16}. 
%In dendrites $[u,v]$ is a topological copy of the usual arc. 
%It is not true, in general, for dendrons. Indeed, note that any connected linearly ordered compact, 
%in its interval topology, space is an example of a dendron. 
%!!! Note also that every linearly ordered compact space is embedded into a connected linearly ordered compact space. 

%0506 
\begin{lem} \label{l:D-prop} 
	Let $D$ be a dendron. 
	
	\ben 
	\item \cite{Mill-Wattel} \  Then $D$ is locally connected and the intersection 
	of arbitrary family of subcontinua of $D$ is either empty or is a continuum. 
	\item \cite[Cor. 2.15.1]{Mill-Wattel} \  $[u,v]$ is the smallest subcontinuum of 
	$D$ containing $u$ and $v$. That is, 
	$$
	[u,v]=\cap \{C \subseteq D: \ u,v \in C \ \text{and} \ C \ \text{is a subcontinuum}\}
	$$
	\item \cite{Bank13} \ 
	%(or, Lemma \ref{l:PropPretree}(A3) below) 
	$[a,b] \subseteq [a,c] \cup [c,b]$ for every $a,b,c \in D$. 
	\een 
\end{lem}
%\begin{proof} (2) %By (1)
%	Q-interpretation is the same as K-interpretation in terms of Bankston \cite{Bank13}. Now see axiom R6 
%	in \cite{Bank13}. 
%\end{proof}

%2606 
Every dendron with its standard betweenness relation is a \textit{pretree} (Section \ref{s:pretrees}). This provides another explanation of \ref{l:D-prop}.3. 
The following proposition can be derived from a result of Bankston \cite[Theorem 3.1]{Bank15},
which, in fact, shows that this assertion holds for every locally connected continuum.  
The direct proof given below, for dendrons, was explained to us by Nicolas Monod.

\begin{prop} \label{p:Monod} 
	Let $D$ be a dendron. Then the 
	betweenness relation  $R_B$ on $D$ is closed. % in $D^3$. 
\end{prop}
\begin{proof}
	Suppose that in $D$ we have converging nets,  $\lim u_i =u, \lim_i w_i =w, \lim v_i =v$ 
	such that $\langle u_i,w_i,v_i \rangle$ for every $i \in I$. %That is, $w_i$ is between $u_i, v_i$ in $D$. 
	We have to show that $\langle u,w,v \rangle$. 
	Assuming %$w$ is not between $u,v$. 
	the contrary, we have $w \notin [u,v]$. Since $[u,v]$ is compact there exist 
        neighborhoods $U$, $W$, $V$ of $u,w,v$ respectively such that $W \cap (U \cup [u,v] \cup V) =\emptyset$. 
	Since $D$ is locally connected, we can assume, in addition, that $U$ and $V$ are connected and closed.  
	There exists $i_0$ such that $w_{i_0} \in W, u_{i_0} \in U,v_{i_0} \in V$. 
%	By our assumption $w_{i_0} \in [u_{i_0},v_{i_0}]$.
%	In contrast we will show that $w_{i_0} \notin [u_{i_0},v_{i_0}]$. Indeed, 
%2409
	By Lemma \ref{l:D-prop}.2 we have 
	$[u_{i_0},u] \subset U, [v,v_{i_0}] \subset V$. 
	Then $[u_{i_0},v_{i_0}] \subset  [u_{i_0},u] \cup [u,v] \cup [v,v_{i_0}]$ by Lemma \ref{l:D-prop}.3. By our choice $W$ does not meet $[u_{i_0},u] \cup [u,v] \cup [v,v_{i_0}]$,  hence, $w_{i_0} \notin [u_{i_0},v_{i_0}]$. 
	This contradiction completes the proof.
\end{proof}

Note that in the \textit{comb space} (which is not locally connected) the relation $R_B$ is not closed (Bankston  \cite[Exercises 3.4(ii)]{Bank15}). 

%2506 
%Below we use Proposition \ref{p:Monod} in Corollary \ref{c:fragm}. In the proof of Lemma \ref{l:propM}.3 we use a particular case of Proposition \ref{p:Monod} for the real interval $D=[a,b] \subset \R$ (where the proof of Proposition \ref{p:Monod} is trivial). 

%\begin{lem} \label{l:arcwise} (see for example \cite{Charatoniks})
%	Let $D$ be a dendrite. Then 
%	\ben 
%	\item $D$ is locally connected. 
%	\item $D$ is uniquely arcwise connected. 
%	\item Every connected subset of $D$ is arcwise connected. 
%	
%	\een
%\end{lem}

%For every pair $(u,v)$ of $D^2$ there exists a uniquely defined linearly ordered arc $[u,v] \subset D$. We have the arc-system $D \times D \to 2^D, (u,v) \mapsto [u,v]$. 
%The arc $[a,b]$ admits a natural "separation ordering" (see \cite[p. 10]{Kok}) which is a linear order. 

\sk \sk
\subsection{Monotone functions}

 	\begin{defin} \label{d:monot2} 
 	Let us say that a (\textbf{not necessarily continuous}) map $f: X \to Y$ between two connected topological spaces is:
 	
 	\ben  
 	\item  \textit{B-monotone} 
	%eli
	%(or, simply, \textit{monotone})  
 	if it respects the betweenness relations $R_B$ (from Definition \ref{d:BrelationD}) of $X$ and $Y$. 
	Meaning that $\langle u,w,v \rangle$ implies $\langle f(u),f(w),f(v) \rangle$. 
	%2406 
	It is equivalent to the requirement that $f$ be \textit{interval preserving}: $f[u,v] \subseteq [f(u),f(v)]$. 
 	 Notation: $f \in M_B(X,Y)$. For $Y=\R$ we write $ M_B(X)$.  
 	 	\item 
 	 \textit{C-monotone} 
	 %eli
	 (or, simply, \textit{monotone})  
	 if the preimage $f^{-1}(A)$ of every connected subset $A \subset Y$ is connected.
 	 Notation: $f \in M_C(X,Y)$. For $Y=\R$ we write $M_C(X)$. 
 	\een
 \end{defin}

For continuous maps on continua definition (2) is well known. 
%If $X$ is compact and $f$ is continuous then $f$ is C-monotone iff the fibers $f^{-1}(y)$ are connected for every $y \in Y$. 
See, Kuratowski \cite[Section 46]{Kur}. 
%2808  If $X$ is a connected space without separating points then any function $X \to Y$ is B-monotone. It follows that 
Not every B-monotone continuous function is C-monotone. 
For a concrete example consider the distance (continuous)  % (being Lipshitz-1). 
 function 
	$$f: [0,1]^2 \to \R, x \mapsto d(x,K), \ K:=([0,\frac{1}{3}] \cup [\frac{2}{3},1]) \times [0,1].$$ 
The fiber $f^{-1}(0)=K$ is not connected. So $f$ is not monotone. 
On the other hand, $f$ is B-monotone because $[0,1]^2$ has no separating points. 

\sk

\begin{lem} \label{l:propM} \ Let $X$ be a connected space. 
	\ben 
	\item Composition of B-monotone (C-monotome) functions is B-monotone (resp., C-monotone). 
	%	\item 
	\item Let $G \curvearrowright X$ be an action of a group $G$ on $X$ by homeomorphisms. 
	For every $g \in G$ and every $f \in$ $M_B(X)$ ($f \in$ $M_C(X)$) we have $fg \in$ $M_B(X)$ (resp., $fg \in$ $M_C(X)$), 
	%2808  
	where $(fg)(x):=f(gx)$. 
	
	\item The set 
	$M_B(X,D)$ is a pointwise closed (hence, compact) subset of $D^X$ for every dendron $D$. 
	\item The set 
	$M_B(X,[c,d])$ is a pointwise closed (hence, compact) subset of $[c,d]^X$. 
	
	\een 
\end{lem}
\begin{proof}
	(1) and (2) are straightforward. 
	
	%	(2) Observe that if $w$ separates $u,v$ in $X$ then $gw$ separates $gu, gv$. So, $<u,w,v>$ implies 
	%	$<gu,gw,gv>$. Hence, $<f(gu),f(gw),f(gv)>$ in $[c,d]$.
	
	%0705 
	(3) 
	Let $f: X \to D$ be the pointwise limit of the net $f_i: X \to D$, $i \in I$ where each 
	%eli
	%$f_i \in M_B(X)$. 
	$f_i \in M_B(X, D)$. 
	We have to show that $f$ is also B-monotone. That is, \nl $\langle u, w, v \rangle$ in $X$ implies that 
	$\langle f(u), f(w), f(v) \rangle$ in $D$. 
	Since every $f_i$ is B-monotone we have $\langle f_i(u),f_i(w),f_i(v) \rangle$. 
	As we already mentioned (Proposition \ref{p:Monod}) the betweenness relation $R_B$ is closed in $D^3$. So,	since $f$ is the pointwise limit of $f_i$ we get $\langle f(u), f(w), f(v) \rangle$.

	(4) is a particular case of (3). 
\end{proof}

\begin{lem} \label{l:monotonicities} 
	For every dendrons $X, Y$ we have $M_C(X,Y) = M_B(X,Y)$ and $M_C(X) = M_B(X)$. 
\end{lem}
\begin{proof} (1) 
	$M_C(X,Y) \subseteq M_B(X,Y)$. 
	
	Assuming the contrary let $f: X \to Y$ be C-monotone but not B-monotone. Then there exist $u,w,v \in X$ such that $\langle u, w, v \rangle$ but 
	$\neg  \langle f(u), f(w), f(v) \rangle$. %Then necessarily $f(w) \notin \{f(u),f(v)\}$.
	 This means that $w$ separates the points $u,v$ but $f(w)$ does not separate the pair $f(u), f(v)$ in $Y$ and $f(w) \notin \{f(u),f(v)\}$. So, 
	$$
	f(w) \notin C:=[f(u),f(v)]. 
	$$ 
	Since $Y$ is a dendron the generalized arc $C=[f(u),f(v)]$ is connected (Lemma \ref{l:D-prop}). 
	By the C-monotonicity of the function $f$ the preimage $f^{-1}(C)$ is connected in $X$. 
	As 
	%eli
	%we already mentioned 
	$w$ separates $u,v$ in $X$ we have
	%. Therefore, 
	$$
	X \setminus \{w\} = U \cup V,
	$$
	where $U,V$ are 
	% certain 
	disjoint open neighborhoods of $u$ and $v$ respectively. Then 
	$u \in f^{-1} (C) \cap U$ and $v \in f^{-1} (C) \cap V$  are disjoint,  
	nonempty (and open in $f^{-1}(C)$). The union of these subsets is 
	$f^{-1}(C)$ (because $w \notin f^{-1}(C)$). So, 
	we get that $f^{-1}(C)$ is not connected, a contradiction to the fact that $f$ is C-monotone. 
	
	\sk (2) 
	 $M_C(X,Y) \supseteq M_B(X,Y)$. 
	
	Assuming the contrary let $f \in M_B(X,Y)$ such that $f \notin M_C(X,Y)$. 
	Then there exists a connected subset $C \subset Y$ such that $f^{-1}(C)$ is not connected in $X$. 
	So there exist (distinct) $u,v \in f^{-1}(C)$ such that the ``generalized arc" 
	$[u,v]$ (which is connected 
	%2808  
	by Lemma \ref{l:D-prop} 
	because $X$ is a dendron) is not contained in $f^{-1}(C)$. 
	Therefore, there exists $w \in [u,v]$ such that $w \notin f^{-1}(C)$. 
	Then $\langle u,w,v \rangle$ but it is not true that 
	$\langle f(u),f(w),f(v) \rangle$ because 
	$f(u), f(v) \in C   \subset Y \setminus \{f(w)\}$, and $C$ is connected. 
	
	This proves $M_C(X,Y) = M_B(X,Y)$. 
	In order to conclude that $M_C(X) = M_B(X)$ use the linear order equivalence 
	$\R \to (0,1) \subset Y:=[0,1], \ x \mapsto \frac{x}{1+|x|}$. 
\end{proof}

%2405
\begin{remark}
Lemma \ref{l:monotonicities} suggests dropping the subscripts ``C" and ``B" 
and writing simply $M(D_1,D_2)$. We write $CM(D)$ for the 
%eli
set of continuous monotone real valued functions on 
%$X$
$D$.
\end{remark}

%2409 Define fragmentability
%3009 We use only the case of functions $f: X \to \R$ with compact $X$. In this case it is equivalent to PCP. 
%So, I think it is optimal to recall just PCP, which is much more well known and transparent. 
Recall that the set $\F(X)$ of 
fragmented real valued functions on $X$ is a vector space over $\R$. 
%3009 
For the definition and properties of fragmented functions we refer to 
\cite{Me-nz,GM-HNS,GM-rose}. In the present paper we only use fragmentability in the case of real valued 
functions $f: X \to \R$ defined on compact $X$. In this case the 
 fragmentability of $f$ is equivalent to the PCP-property (see \cite{GM-HNS}). 
Meaning that for every closed nonempty subset $Y \subseteq X$ the restriction $f|_Y: Y \to \R$ has a point of continuity.  
 For Polish $X$, $\F(X)$ coincides with
the set $\mathcal{B}_1(X)$ of Baire  class 1 functions. 
%2506
A function $f: X \to Y$ is said to be \emph{Baire class 1 function} if the inverse image $f^{-1}(O)$ of every open set 
$O \subset Y$ is $F_\sigma$ in $X$. 

In \cite{GM-tame, Me-Helly} we proved that every linear order preserving function on a compact linearly ordered topological space is fragmented. 
The following theorem is a result in the same spirit. 

%2409  
% I suggest introducing the following definitions
%3009 (2) is weaker than (1). So "hereditarily stable" for (2) is not suitable. I suggest 
%"stable" for (1) and "almost stable" for (2)
%other options: "weakly stable", ...

\begin{defin} \label{d:stable}
Let $R$ be an abstract ternary relation on a Hausdorff topological space $(X,\tau)$.
\begin{enumerate}
\item
We say that $R$ is \textit{$\tau$-stable} if for every pair of distinct points $u,v \in X$, 
there are,	a point $w \in X$ with $w \in [u,v] \setminus \{u,v\}$, 
	(where $[u,v]:=\{x \in X: \langle u,w,v \rangle\} \cup \{u,v\}$), 
	and neighborhoods $U,V$ of $u,v$ respectively, 
	such that $w \in [x,y]$ for every $x \in U, y \in V$. 
\item
We say that $R$ is 
%3009 \textit{hereditarily $\tau$-stable} 
 %\textit{almost $\tau$-stable} 
% maybe 
 \textit{weakly $\tau$-stable} 
  if for every infinite subset 
$K \subset X$ there exist: a pair of distinct points $u,v \in K$, 
a point $w \in X$ with $w \in [u,v] \setminus \{u,v\}$, 
and neighborhoods $U,V$ of $u,v$ respectively, 
such that $w \in [x,y]$ for every $x \in U \cap K, y \in V \cap K$. 
\end{enumerate}
\end{defin}

%\begin{defin} \label{d:stable} 
%	Let $R$ be an abstract ternary relation on a Hausdorff topological space 
%	$(X,\tau)$. We say that $R$ is \textit{$\tau$-stable} if for every infinite subset 
%	$K \subset X$ there exist: a pair $u,v \in K$, a point $w \in X$ with $w \in [u,v] \setminus \{u,v\}$, 
%	%2808  
%	(where $[u,v]:=\{x \in X: \langle u,w,v \rangle\} \cup \{u,v\}$) 
%%% 
%	and neighborhoods $U,V$ of $u,v$ respectively, such that $w \in [x,y]$ for every $x \in U \cap K, y \in V \cap K$. 
%\end{defin}

Let $R$ be a ternary relation on $X$ and let 
$f: X \to \R$ respect $R$ and the standard betweenness relation of the reals.   
We will say that $f$ is an  R-\textit{monotone} function.

\begin{thm} \label{t:GenFragm} 
	Let $(X,\tau)$ be a compact space and $R$ a ternary relation on $X$ which is 
	%3009 $\tau$-stable. 
	weakly $\tau$-stable. 
	Then every $R$-monotone function $f: X \to \R$ is fragmented. 
\end{thm}
\begin{proof} %It is equivalent to showing this for bounded $f: X \to [0,1]$. 
	%We may assume by Lemma \ref{l:monotonicities}  that $f$ is a B-monotone map.  
	If $f$ is not fragmented then, by \cite[Lemma 3.7]{Dulst} there exists a closed nonempty subspace 
	%eli
	%$L \subset D$ and real numbers $\alpha < \beta$
	$L \subset X$ and real numbers $\alpha < \beta$ 
	% such that 
	such that the subsets 
	$f^{-1}(-\infty,\alpha) \cap L$ and $f^{-1} (\beta,\infty) \cap L$ 
	% of $L$ 
	are dense in $L$. 
	That is, 
	\begin{equation} \label{e:dense} 
	cl(f^{-1}(-\infty,\alpha) \cap L)=	cl(f^{-1} (\beta, \infty) \cap L)=L. 
	\end{equation}
	$L$ is necessarily infinite 
	%2808  
	(because $f^{-1}(-\infty,\alpha) \cap L$ and $f^{-1} (\beta, \infty) \cap L$ are disjoint). 
	Since $X$ is 
	%3009 $\tau$-stable 
	weakly $\tau$-stable 
	we may choose distinct points 
	$u,v \in L$, $w \in X$ with $w \in [u,v] \setminus \{u,v\}$, and $\tau$-neighborhoods 
	$U$ and $V$ of $u$ and $v$ respectively such that $\lan x,w,y \ran$ 
	for every $x \in U \cap L, y \in V \cap L$.
	
	By Equation \ref{e:dense} there exist $u',u'' \in U \cap L$ and $v', v'' \in V \cap L$ such that 
	\begin{equation} \label{eq:green-red} 
	f(u') <\a, f(v')< \a, \ \ \ \beta < f(u''), \beta < f(v'') 
	\end{equation}
	It is enough to show that $f$ is not $R$-monotone. There are two cases 
	%eli
	to check:
	
	\sk 
	
	Case 1. $f(w) \leq \beta$. 
	\sk
	Then $\langle u'',w, v'' \rangle$ but $f(w)$ is not between  $f(u''), f(v'')$ in $\R$. 
	
	\sk
	Case 2. $\beta < f(w)$. 
	\sk
	Then $\langle u',w,v' \rangle $ but $f(w)$ is not between   $f(u'), f(v')$ in $\R$.

\end{proof} 

%!!!
\begin{thm} \label{t:MonotFragm} Let $D$ be a dendron. Every monotone function $f: D \to \R$ is fragmented 
(Baire class 1, when $D$ is a dendrite). 
It follows that $M(D) \subset \F(D)$. 
\end{thm}
\begin{proof}
For a dendron $D$ the standard betweenness relation (Definition \ref{d:BrelationD}) is stable. 
Indeed, by the definition of dendrons for every distinct $u \neq v$ we have a separation by a point $w$. So, $D\setminus \{w\}=U \cup V$, where $U$ and $V$ are open disjoint neighborhoods of $u$ and $v$. Hence, $w$ separates any pair $x \in U, y \in V$.
  
	\end{proof}

For dendrites we have  
$M(D) \subset \mathcal{B}_1(D)=\F(D)$.

\begin{cor} \label{c:tameFamily}
	For every dendron $D$ the family $F:=CM(D,[0,1])$ is tame. 
	%3009 this was explicitly defined now in introduction (that is, $F$ does not contain independent sequences). 
\end{cor}
\begin{proof} 
By Lemma \ref{l:propM}.4, every function $\varphi: D \to [0,1]$ from the pointwise closure of  
	$F$ in $[c,d]^D$ is a (not necessarily continuous) 
	B-monotone function. 
	By Lemma \ref{l:monotonicities}, 
	$\varphi \in M(D,[0,1])$. By Theorem \ref{t:MonotFragm} we know that $M(D,[0,1]) \subset \F(D)$ 
%Hence, we
%can 
and we
conclude that $cl_p(F) \subset \F(D)$. This means that $F$ is a Rosenthal family, in terms of \cite{GM-rose}. 
%It is equivalent 
This is the same as saying 
that $F$ does not contain an independent sequence (for a detailed proof see for example \cite[Theorem 2.12]{GM-tame}). 
\end{proof}

%eli
%1804 
%\begin{remark} \label{r:Mill}  
%	In \cite[Cor. 2.15]{Mill-Wattel} van Mill and Wattel proved the following remarkable result:  
%	for every dendron $D$ and every subcontinuum $C$ there exists 
%	a naturally defined continuous retraction $r_C: D \to C$. 
%	
%	Moreover, as we 
%	%eli
%	were
%	informed by J. van Mill,
%	it is easy to check that this retraction is always monotone. 
%	For dendrites these facts are well known, \cite[Theorem 1.2]{Charatoniks}. 
%	% and in this case the corresponding map is the so-called \textit{first point map}   
%\end{remark}

%0306
In \cite[Cor. 2.15]{Mill-Wattel} van Mill and Wattel proved the following remarkable result.
We thank Jan van Mill for providing us %with 
the short proof below. 
	For dendrites Theorem \ref{r:Mill} is well known, \cite[Theorem 1.2]{Charatoniks}. 

\begin{thm} \label{r:Mill} %\cite[Cor. 2.15]{Mill-Wattel}
For every dendron $D$ and every subcontinuum $C$ there exists a naturally 
defined continuous retraction $r_C: D \to C$. 
Moreover, this retraction is always monotone.
\end{thm}

\begin{proof}   
The map $r_C: D \to C$ is defined by 
$$
r_C(x)= \cap \{[x,c] \cap C: \ c \in C\}. 
$$
%2808 
	 We discuss only the monotonicity of $r_C$. Other details see in \cite[Cor. 2.15]{Mill-Wattel}. 
Let $x \in D$ and $y=r_C(x)$. 
If $p \in [x,y]$ then $[p,y]$ is contained in $[x,y]$. Since $y\in C$, 
the formula for $r_C(p)$ gives us that $r_C(p) \in [p,y] \cap C = {y}$, hence $r_C(p) =y$. 
%eli
%That $[p,y] \cap C = \{y\}$ is clear since $r_C(p)$ is the ``nearest point". 
Thus all the points in the fiber of the point $y\in C$ can be connected to $y$ by a continuum in the fiber. 
%2808  
So, $r_C^{-1}(y)$ is connected for every $y \in C$. Since $r_C: D \to C$ is a continuous closed map then 
by Section 46, subsection I, Theorem 9 in \cite{Kur} the map $r_C$ is C-monotone (Definition \ref{d:monot2}). 
The subcontinuum $C$ is also a dendron. So, $r_C$ is also $B$-monotone (Lemma \ref{l:monotonicities}). 
\end{proof}

This result leads to the following 
  
\begin{lem} \label{l:retracts}
	On every dendron $D$ and 
%li
%a 
every pair $u,v$ of distinct points in $D$ there exists a continuous 
%2405 C-monotone (equivalently, B-monotone) 
monotone 
function $f: D \to [0,1]$ such that $f(u)=0, f(v)=1$. 
%Hence, $BV_1(D,[0,1]) \cap C(D,[0,1])$ separates points of $D$. 
\end{lem}
\begin{proof} By Theorem \ref{r:Mill} 
	for every pair of distinct points $u,v \in D$ 
	we have a monotone continuous retraction $r_{[u,v]}: D \to [u,v]$. 
	Now recall that the ``generalized arc" $[u,v]$ is a linearly ordered compact 
	connected space, \cite{Mill-Wattel}. 
%2808  (it is an arc when $D$ is metrizable; i.e. a dendrite).
	 By results of Nachbin \cite[p. 48 and 113]{Nach} we have an 
	 order-preserving (hence, monotone in the sense of Definition \ref{d:monot2}.1) continuous map 
%eli
$h : [u, v] \to [0,1]$.
	%$D \to [0,1]$  which separates $u,v$. 
	The  composition $f =  h \circ r_{[u,v]}$ 	
	 is the required continuous monotone map $D \to [0,1]$ which separates $u$ and $v$. 
\end{proof}

%0606 
%Now we are ready to prove the  main theorem formulated in the introduction. 

\begin{thm} \label{t:Dendron} 
	Let $D$ be a dendron. For every topological group $G$ and continuous action 
	$G \curvearrowright D$, the dynamical $G$-system $D$ is Rosenthal representable (hence, also tame). 
	It follows that the topological group $H(D)$ is Rosenthal representable.  
\end{thm}
\begin{proof} 
	 By Lemma \ref{l:propM}.2 we have $fg \in CM(D,[0,1])$ 
	for every $g \in G$ and every $f \in CM(D,[0,1])$. So, if $F:=CM(D,[0,1])$ then $FG=F$ is a $G$-invariant bounded family of continuous functions.
	By Corollary \ref{c:tameFamily}, $F$ is a tame family. 
	 By Lemma \ref{l:retracts}, $CM(D,[0,1])$ separates the points of $D$. 
	So one may apply Theorem \ref{t:criter1} and we obtain that the dynamical $G$-system $D$ is Rosenthal representable.  
	
	%2808  
	Finally note that it is straightforward to see that Rosenthal representability of any compact dynamical $H(K)$-system $K$ implies that the topological group $H(K)$ is Rosenthal representable (for details see for example \cite[Lemma 3.5]{GM-tame}).  

\end{proof}

%2808  From Theorem \ref{t:MonotFragm} and Lemma \ref{l:retracts} (using \cite[lemma 2.3.3]{GM-rose}) we get the following generalization of Theorem \ref{t:MonotFragm} 

\begin{thm} \label{t:MonotFr} 
	Every monotone map $f: D_1 \to D_2$ between dendrons is fragmented (Baire class 1 map, if $D_1,D_2$ are dendrites). 
\end{thm}
 %2808  adding a proof 
\begin{proof}
	By Lemma \ref{l:retracts} there exists a family of continuous monotone maps $\{q_i: D_2 \to [0,1]: i \in I\}$ which separates the points of $D_2$. Then any composition  $q_i \circ f: D_1 \to [0,1]$ is monotone, hence fragmented by Theorem \ref{t:MonotFragm}. Now, using \cite[lemma 2.3.3]{GM-rose} we obtain that the original map $f: D_1 \to D_2$ is also fragmented.  
	\end{proof}

The tameness of any continuous group actions on dendrons 
$G \curvearrowright D$ can be derived 
%2808  directly from Theorem \ref{t:MonotFr}.
from Theorem \ref{t:MonotFr} using the following 

\begin{cor} \label{c:fragm} 
Let $G$ act on a dendron $D$ then every element of the enveloping semigroup $E(G, D)$ is
fragmented, hence the system $(G, D)$ is tame.
\end{cor} 
\begin{proof}
%2506 Indeed, 
Using Lemma \ref{l:propM}.3 we obtain that  
every element $p \in E(G, D)$ is a monotone map $p: D \to D$. 
By Theorem \ref{t:MonotFr}, $p$ is a fragmented map and it follows that
%So, $p$ is fragmented by Theorem \ref{t:MonotFr}. This means that 
the $G$-system $D$ is tame by the enveloping semigroup characterization of
%eli
tameness, \cite{GM-rose}.  
\end{proof}

Theorem \ref{t:Dendron} implies also the following purely topological nontrivial fact. 

\begin{cor} \label{c:DisWRN} 
	Every dendron $D$, as a topological space, is Rosenthal representable. That is, $D$ is WRN. 
\end{cor}

%Every linearly ordered space
%%2409 
%% is 
%can be
%embedded into a connected compact linearly ordered space (which is a dendron). 
%Therefore, every linearly ordered space is WRN. 

As a related result note that by \cite[Theorem 6.6]{Mill-Wattel} 
a Hausdorff space can be embedded in a dendron if and only if it possesses a \textit{cross-free} 
(see \cite{Mill-Wattel} for definitions) 
closed subbase. 

%0306
\begin{remark} \label{r:semigr} 
%	\ben 
%	\item Theorem \ref{t:Dendron} implies in particular that $W$ as a dynamical $(H(W)$-system is Rosenthal representable (hence, tame) for the Wazewski dendrite $W$. Hence also, the topological group $H(W)$ is Rosenthal representable.  
%	
%	Recently Kwiatkowska gave a description of the universal minimal $G$-system $M(G)$ for the topological group $G=H(W)$.  
%	It would be interesting to check if $M(G)$ as a dynamical $G$-system is tame. This space $M(G)$ is not a dendrite (being zero-dimensional) but perhaps it admits some suitable betweenness relation.  
%	
%	\item It is still unclear what can be said about 
%	the Lelek fan $L$ (treelike but not a dendrite). Is it true that $L$ is tame as an $H(L)$-system ?
%	
%	\item 
Theorem \ref{t:Dendron} remains true also for continuous \textit{monotone} monoid actions 
%of 
$S$ on $D$
%2409
%. Meaning, under the condition 
such 
that for all $s \in S$ the corresponding $s$-translation $D \to D$ is monotone. 
Clearly continuous group actions on dendrons are always monotone.   
%	\item 
%	It is well known that every invertible cascade on a dendrite has zero entropy, \cite{Chinen, Kato}. 
%	This fact can be derived from theorem \ref{t:Dendron} taking into account that tameness implies zero entropy for cascades by a result of Kerr and Li. 
%
%	\een 
\end{remark}

\sk
 
\section{Monotone actions on median pretrees} 
\label{s:pretrees} 

In this section we consider actions on a pretree; 
a useful treelike structure that naturally generalizes several important structures
%2309  
 including linear orders and the betweenness relation on dendrons. 
By a \textit{pretree} (see for example \cite{B,Mal-14}) we mean a pair $(X,R)$, where 
$X$ is a set and $R$ is a ternary relation on $X$
(we write $\lan a,b,c \ran$ to denote $(a,b,c) \in R$) satisfying the following three axioms:
\begin{itemize}
	\item [(B1)] $\lan a,b,c \ran \Rightarrow \lan c,b,a \ran$. 
	\item [(B2)] $\lan a,b,c \ran \wedge \lan a,c,b \ran \Leftrightarrow b=c$.
	\item [(B3)] $\lan a,b,c \ran \Rightarrow \lan a,b,d \ran  \vee \lan d,b,c \ran$.
\end{itemize} 
In \cite{AN} such a ternary relation is called a \textit{B-relation}. 

It is convenient to use also an interval approach. For every $u,v \in X$ define 
$$
[u,v]:=\{x \in X: \lan u,x,v \ran\}
$$ 

In the list of properties below the first four conditions (A0),(A1),(A2),(A3), 
as a system of axioms, is equivalent to the above definition via (B1), (B2), (B3)
(see \cite{Mal-14}).  

\begin{lem} \label{l:PropPretree} \ 
	In every pretree $(X,R)$ we have 
	\begin{itemize}
	\item [(A0)] $[a,b] \supseteq \{a,b\}$. 
	
	\item [(A1)] $[a,b]=[b,a]$. 
		\item [(A2)] If $c \in [a,b]$ and $b \in [a,c]$ then $b=c$.  
		
		\item [(A3)]  $[a,b] \subseteq [a,c] \cup [c,b]$ for every $a,b,c \in X$.

	\item [(A4)] $[a,b] = [a,c] \cup [c,b]$ for every $a,b \in X, c \in [a,b]$. 
	\item [(A5)]  If $b \in [a,c]$ and $c \in [a,d]$ then $c \in [b,d]$. 

	\end{itemize}  
\end{lem}

%Every linearly ordered set defines a pretree.  
Following \cite{Mal-14} we define the so-called \textit{shadow topology} $\tau_s$ on $(X,R)$. 
%In order to define $\tau_s$ for $(X,R)$ 
Given an ordered pair $(u,v) \in X^2, u \neq v$,
let 
$$
S^v_u:=\{x \in X: u \in [x,v]\}
$$
be the \textit{shadow} in $X$ defined by the ordered pair $(u,v)$. 
%Roughly 
Pictorially, the shadow $S^v_u$ is 
%0406
%discarded 
cast
by a point $u$ when the light source is located at the point $v$. 
The family $\mathcal{S} = \{S^v_u: u,v \in X, u \neq v\}$ is a subbase for the closed sets
%\textit{closed subbase} 
of the topology $\tau_s$. 
%!!!
%The complement 
%	$$v \in X \setminus S^v_u=\{x \in X: u \notin [x,v]\}$$
%is a subbase neighborhood of $v$. 

 In the case of a linearly ordered set we get the interval topology. 
In general, for an abstract pretree the shadow topology is 
%frequently 
often 
(but not always) Hausdorff. 
Furthermore, by \cite[Theorem 7.3]{Mal-14} a pretree equipped with its shadow topology is Hausdorff 
if and only if, as a topological space, it can be embedded into a dendron. 
%3009 PROMOTION  Hence, every Hausdorff pretree (e.g., linearly ordered topological space) is WRN by Corollary \ref{c:DisWRN}. 
So, by Corollary \ref{c:DisWRN} we can deduce the following: 
\begin{cor}
Every Hausdorff pretree (e.g., linearly ordered topological space) is a WRN topological space. 
\end{cor}

\sk
 
For every triple $a,b,c$ in a pretree $X$ the {\it median} $m(a,b,c)$ is the intersection 
$$
m(a,b,c):=[a,b] \cap [a,c] \cap [b,c]. 
$$
When it is nonempty the median is a singleton (see for example, \cite[p. 14]{B}). 
A pretree $(X,R)$ for which this intersection is always nonempty is called a \textit{median pretree}. 

%2409  I think that we need to define here what is a median algebra

%3009 
A {\it Median algebra} (see for example \cite[p. 14]{B} or, \cite{Vel}) is a pair $(X,m)$, 
where the function $m: X^3 \to X$ satisfies the following three axioms:
\begin{itemize}
	\item [(M1)]  $m(x,x,y)=x$.  
	\item [(M2)]  $m(x,y,z)=m(x,y,z)=m(y,z,x)$.
	\item [(M3)]  $m(m(x,y,z),u,v)=m(x,m(y,u,v),m(z,u,v)).$
\end{itemize}

\begin{remark} \label{r:med} \ 
	\ben 
	\item Every median pretree is a \textit{median algebra}.  
	%(for the axioms of median algebra see \cite[p. 14]{B} or, \cite{Vel}). 
	\item A map $f: X_1 \to X_2$ between two median algebras is monotone (i.e., interval preserving) 
	%2309
	if and only if $f$ is median-preserving (\cite[page 120]{Vel}) if and only if $f$ is convex (\cite[page 123]{Vel}). Convexity of $f$ means that the preimage of a convex subset is convex. 

	\item Every median pretree is Hausdorff %(and normal) 
	in its shadow topology (\cite[Theorem 7.3]{Mal-14}). 
	\een 
\end{remark}

A \textit{compact (median) pretree} is a (median) pretree $(X,R)$ for which the shadow topology  $\tau_s$ is compact.

%The following examples of median pretrees can be found in \cite{Mal-14}. 

\begin{examples} \label{ex:median} \ 
	\ben 
	\item %2808  It is well known (see for example, \cite{Mal-14}) that 
	Every dendron $D$ is a compact median pretree with respect to the standard betweenness relation $R_B$. 
		Its shadow topology is just the given compact Hausdorff topology on $D$ (see \cite{Mal-14,Mill-Wattel}). 
	
	%A pseudotree is a partially ordered set such that every $a_{\downarrow}: =\{x \in X: x \leq a\}$ is linearly ordered.
	
	\item Every linearly ordered set is a median pretree. Its shadow topology is just the interval topology of the order. 
%	\item  Every meet-semilattice is a 
%	%2808  pretree (with the pretree structure $\lan u, w, v \ran$ whenever one of the following conditions is satisfied: $w = u \wedge v, w \leq u, w \leq v$). Such a pretree has the median:  
%	median algebra with the median  
%	$$
%	m(a,b,c):=\max\{a\wedge b, b \wedge c, c \wedge a\}. 
%	$$
%	%Very particular case of this is a linearly ordered set. 
	\item Let $X$ be a \textit{$\Z$-tree} (a median pretree with finite intervals $[u,v]$). 
	Denote by $Ends(X)$  the set of all its ends. 
	According to \cite[Section 12]{Mal-14} the set $X \cup Ends(X)$ carries a natural 
	$\tau_s$-\textit{compact} median pretree structure.
	\een
\end{examples}

%\footnote{Though this retraction itself is very familiar, I could not find this lemma (the continuity) in the literature. I don't know if it is known.}
\begin{prop} \label{p:mRetr} 
	Let $(X,R)$ be a median pretree. Then the retraction map 
	$$
	\phi_{u,v}: X \to [u,v], \ x \mapsto m(u,x,v)
	$$
	is monotone and continuous in the shadow topology for every $u, v \in X$. 
\end{prop}
\begin{proof} 
Note that always $c \in [a,b]$ if and only if $med(a,c,b)=c$. So, $\phi_{u,v}(x)=x$ for every $x \in [u,v]$. 
This means that $\phi_{u,v}$ is a retraction.
	%2808 A well known property of median algebras, namely \cite[Equation 8.7]{Shol} implies that every $\phi_{u,v}$ is a median preserving map (hence, monotone). 
	A well known property of median algebras, namely \cite[Equation 8.7]{Shol}, directly implies that
	$$
	m(m(u,x_1,v), m(u,x_2,v), m(u,x_3,v)) = m (u,m(x_1,x_2,x_3),v). 
	$$ 
	This means that $m(\phi(x_1), \phi(x_2), \phi(x_3))= \phi(m(x_1,x_2,x_3))$.  
	So every $\phi_{u,v}$ is a median preserving map (hence, monotone, as a map between pretrees). 
	%%%%
	
	Now we check the continuity of $\phi_{u,v}$. 
	If $u=v$ then $\phi_{u,v}$ is constant. So, we can suppose that $u \neq v$.
	Every interval $[a,b]$ is a convex subset of $X$. Hence its interval topology and its topological subspace topology of $(X,\tau_s)$ are the same 
	(see \cite[Prop. 6.5]{Mal-14}).  
	 It is enough to show that the preimage of a closed subbase 
	 %2409
	 element in the space $[u,v]$ is closed in the shadow topology. We prove that in fact 
$$	\phi_{u,v}^{-1}[u,w] = S^v_w \ \ \ \forall \ w \in [u,v) \ \ \text{and} \ \  \phi_{u,v}^{-1}[w,v] = S^u_w \ \ \ \forall \ w \in (u,v].$$
First we show that $\phi_{u,v}^{-1}[u,w] \subseteq S^v_w$.

Let $x \in \phi_{u,v}^{-1}[u,w]$. This means that $m:=m(u,x,v) \in [u,w]$. Since %$m \in [u,w]$, 
$w \in [u,v]$, by (A5) we have $w \in [m,v]$. So, $[m,w] \cup [w,v]=[m,v]$ by (A4). Again, by (A4) we obtain $[x,m] \cup [m,v]=[x,v]$. So, $w \in [x,v]$. Hence, $x \in S^v_w$. 

Now we show $\phi_{u,v}^{-1}[u,w] \supseteq S^v_w$.
Let $x \in S^v_w$. This means that $w \in [x,v]$. We have to show that $m \in [u,w]$ (where $m:=m(u,x,v)$).
%2406
 Assuming the contrary, suppose $m \notin [u,w]$. 
 Clearly, $m \in [u,v]$. By (A3), $[u,v] \subseteq [u,w] \cup [w,v]$. So, we have $m \in (w,v]$. Since $w \in [v,x]$ and $m \in [v,w]$, by (A5) we get $w \in [m,x] = [x,m]$. Similarly, since $m \in [v,w]$ and $ w \in [v,u]$, by (A5) we get $w \in [m,u]$.
 
 %2808  
 Now taking into account that $m \neq w$ 
by (A2) we have $m \notin [x,w]$ and also $m \notin [w,u]$. So,  $m \notin [x,w] \cup [w,u]$. 
Then (A3) guarantees that $m \notin [x,u]$. This contradicts the fact that $m$, being the median of the triple $x,u,v$, belongs to $[x,u]$.    
\sk 
Similarly we can check also that $\phi_{w,v}^{-1}[u,w] = S^u_w$.

	\end{proof}

%2309  proposition promoted to thm
\begin{thm} \label{p:ConvFun} 
	%2309   
	Let $X$ be a median pretree. Then every pair of monotone (equivalently, convex) 
	real valued functions $f_i: X \to \R$, $i \in \{0,1\}$ is not independent. 
\end{thm}
\begin{proof}  Assuming that $\{f_1, f_2\}$ is an independent pair, there exist real numbers $a < b$ 
	such that 
	$$
	A_1 \cap A_2 \neq \emptyset, \ A_1 \cap B_2 \neq \emptyset, \ B_1 \cap A_2 \neq \emptyset, \ B_1 \cap B_2 \neq \emptyset, \
	$$
	where $A_i=f_i^{-1}(-\infty,a), \ B_i=f_i^{-1}(b, \infty)$, $i \in \{0,1\}$. 
	
	Choose four points $$a \in A_1 \cap B_2, \ b \in A_1 \cap A_2, \ c \in A_2 \cap B_1, \ d \in B_1 \cap B_2.$$
	By the monotonicity of the functions $f_1, f_2$ the preimages of convex subsets are convex. So, 
	$A_i, B_i$ are convex. Then $$[a,b] \subset A_1, [b,c] \subset A_2, [c,d] \subset B_1, [d,a] \subset B_2.$$
	
%	$$
%	[a,c] \subset [a,b] \cup [b,c] \subset A_1 \cup A_2
%	$$ 
%	$$
%	[c,a] \subset [c,d] \cup [d,a] \subset B_1 \cup B_2.
%	$$
	Consider the median $m:=med(c,a,d)=[a,c] \cap [c,d] \cap [d,a]$. Then 
	$$m \in [a,c] \cap B_1 \cap B_2.$$  
	Using (A3) we get
	$$
	[a,c] \subset [a,b] \cup [b,c] \subset A_1 \cup A_2.
	$$
	Therefore, $$m \in (A_1 \cup A_2) \cap (B_1 \cap B_2).$$
	Clearly, $A_1 \cap B_1 = A_2 \cap B_2 =\emptyset$. So, $(A_1 \cup A_2) \cap (B_1 \cap B_2)= \emptyset$, a contradiction. 
	
	\end{proof}

\sk

For a compact median pretree $X$ we denote by $H_+(X)$ the topological group of 
$R$-monotone (equivalently, median-preserving) homeomorphisms. 
We treat $H_+(X)$ as a topological subgroup of the full homeomorphism group $H(X)$.

The following result generalizes Theorem \ref{t:Dendron}.  
In the case of a dendron $D$ we have $H_+(D)=H(D)$. 
%0506
\begin{thm} \label{t:pretree} 
	For every compact median 
	pretree $X$ and its automorphism group $G=H_+(X)$ the action 
	of the topological group $G$ on $X$ is Rosenthal representable. 
	It follows that the topological group $H_+(X)$ is Rosenthal representable.
\end{thm}  
\begin{proof} 
	%3009 
	Recall again that a median pretree is Hausdorff 
	%(and normal) 
	in its shadow topology (Remark \ref{r:med}.3). 
	%%  
	%2808 
	By Proposition \ref{p:mRetr} the retraction map 
	$$
	\phi_{u,v}: X \to [u,v], \ x \mapsto m(u,x,v)
	$$
	is monotone and continuous in the shadow topology for every $u, v \in X$.
	%$\phi_{u,v}(u)=u, \phi_{u,v}(v)=v.$ 
	%% 
	Every $[u,v]$ is a linearly ordered set with respect to the order: $x \leq y$ 
	whenever $\lan x,y,v \ran$ (see for example, \cite[p.14]{B}). 
	Moreover $[u,v]=\phi(X)$ is 
	compact in the subspace topology which coincides with the interval topology of the linear order. 
	So, as in Lemma \ref{l:retracts}, using Nachbin's result we see that the set $CM(X)$ of continuous monotone functions separates the points. Also, $CM(X)$ is $G$-invariant because the action $G \curvearrowright X$ is monotone. Theorem \ref{p:ConvFun} guarantees that  $CM(X)$ is a tame family.  
	A median pretree is always Hausdorff %(and normal) 
	in its shadow topology (Remark \ref{r:med}.3). %by \cite[Theorem 7.3]{Mal-14}. 
	%!!!!!!!!!!!
	The rest is similar to the proof of Theorem \ref{t:Dendron}.
\end{proof}

 By Example \ref{ex:median}.4, Theorem \ref{t:pretree} 
 applies when $X$ is a $\Z$-tree and we get the following:

\begin{cor} \label{c:Ends} 
	Let $X$ be a $\Z$-tree. Denote by $Ends(X)$ the set of all its ends.  
	Then for every monotone group action 
	$G \curvearrowright X$ with continuous transformations 
	the induced action of $G$ on the compact space $\widehat{X}:=X \cup Ends(X)$ 
	is Rosenthal representable.  
\end{cor}

	Such compact spaces $\widehat{X}$ as in Corollary \ref{c:Ends} are  often zero-dimensional. 
So, at least, formally this case cannot be deduced from the dendron's case.

\br
%2309 
\subsection{Generalized Helly's selection principle} 

\begin{thm} \label{GenHellyThm} %(Generalized Helly's selection theorem) 
Let $X$ be a median pretree and $\{f_n: X \to \R\}_{n \in \N}$ be a bounded sequence of convex  
(equivalently, monotone) real valued maps. Then there exists a pointwise converging subsequence. 
\end{thm}
\begin{proof} 
Combine Theorem \ref{p:ConvFun} 
%(which asserts that every pair of of real functions on $X$ is not independent where $X$ is a median pretree) 
with the following form of Rosenthal's theorem: 
Let $f_n: S \to \R$ be a bounded sequence of functions on a set $S$. 
Then it has a subsequence which is pointwise converging or 
%  (``not strong alternative") ???
has a subsequence which is independent;  
see Theorem 1 in Rosenthal's classical paper \cite{Ros0}. 
In Rosenthal's original formulation he states a weaker statement ``$l_1$-subsequence" (instead of ``independent"). 
However, as Rosenthal's proof shows he proves in fact a little bit more 
(see text above Lemma 5 on page 2413),
%. Actually, he shows 
namely
that there exists an 
independent subsequence. 

\end{proof}

\begin{cor} \label{c:Helly} 
Let $X$ be either a dendron or a linearly ordered set. 
Then the pointwise compact set of all monotone maps $M(X,[c,d])$ into the real interval $[c,d]$ is sequentially compact. 
\end{cor}

%18
%\begin{remark}
%For linearly ordered sets Theorem \ref{GenHellyThm} can be extended to 
%sequences of real valued functions with bounded total variation \cite{Me-Helly}. 
%This suggests 
%%1810 that one should search 
%an analog for bounded variation functions, defined on dendrons, or more generally on (median) pretrees. 
%Namely defining the variation as the supremum of all variations for the restrictions on 
%the linearly ordered (generalized) arcs $[u,v]$. 
%%1810 
%In this case one may show that Theorem \ref{GenHellyThm} can be extended to a sequence of functions with total bounded variation defined on median pretrees. 
%\end{remark} 

%2a
\begin{remark}
	For linearly ordered sets Theorem \ref{GenHellyThm} can be extended to 
	sequences of real valued functions with bounded total variation \cite{Me-Helly}. 
	This suggests the idea that one should search for a right analog for bounded variation functions, defined on dendrons, or more generally on (median) pretrees. 
\end{remark}

  \br 
  
%3009 \subsection{Stability of Hausdorff pretrees}
\subsection{Fragmentability of monotone functions on compact Hausdorff pretrees}

\begin{lem} \label{l:stabPretr} 
	If the shadow topology $\tau_s$ on a pretree $(X,R)$ is Hausdorff (e.g., median pretree). Then the pretree relation $R$ is 
	%3009 
	weakly 
	$\tau_s$-stable (Definition \ref{d:stable}.2).   
\end{lem}
\begin{proof}
	We have to show that for every infinite subset $K \subset X$ there exist: 
	a pair $u,v \in K$, a point $w \in X$ with $w \in [u,v] \setminus \{u,v\}$, 
	neighborhoods $U,V$ of $u,v$ respectively, 
	such that $w \in [x,y]$ for every $x \in U \cap K, y \in V \cap K$. 
	In fact, we show this for every $x \in U, y \in V$. 
	
	First of all note that there exist distinct $u,v \in K$, $w \in X$ 
	such that $w \in [u,v] \setminus \{u,v\}$. If not then 
	$K$ is a \textit{star subset} in terms of \cite{Mal-14}. Since $\tau_s$ is 
	Hausdorff there is no \textit{infinite} star subset in $X$ by \cite[Theorem 7.3]{Mal-14}.
	
	Consider the standard shadow topology neighborhoods 
	$$u \in U:=X \setminus S^u_w \  \text{and} \ v \in V:=X \setminus S^v_w.$$ 
	We are going to show that 
	$w \in [x,y]$ for every $x \in U, y \in V$. 
	%Consider the medians
	%$$
	%m_1:=m(x,u,w), \ \ \ m_2:=(y,v,w). 
	%$$
	By (B3) and $\lan u,w,v \ran$ we have $\lan u,w,x \ran  \vee \lan x,w,v \ran$. 	
	By our choice of $x \in U=X \setminus S^u_w$ it is impossible that $w \in [x,u]$. 
	So we necessarily have the second condition $\lan x,w,v \ran$. Now apply 
	again (B3) but now for the triple $\lan x,w,v \ran$ and a point $y \in V$. 
	Then by the definition of $V:=X \setminus S^v_w$ it is impossible that $\lan y,w,v \ran$. 
	Therefore, we necessarily have $\lan x,w,y \ran$. 
	
\end{proof}

Lemma \ref{l:stabPretr} and Theorem \ref{t:GenFragm} imply

\begin{thm} \label{t:FragmPretr} 
	Let $(X,\tau_s)$ be a compact Hausdorff pretree (e.g., median pretree). 
	Then every monotone real valued function $f: X \to \R$ is fragmented. 
\end{thm}

\begin{remark}
	By a result of Malyutin \cite[Theorem 7.3]{Mal-14} 
	every $\tau_s$-Hausdorff pretree $X$ can be topologically embedded into a dendron $D$. 
	This leads to a natural question if there exists an \textit{equivariant version} of this embedding.  
	Consider the automorphism group $G:=\Aut(X,R)$ 
	(of all monotone homeomorphisms) as a topological subgroup of $H(X, \tau_s)$. 
	Is it true that there exists an equivariant embedding of $(G,X)$ into $(H(D),D)$ 
	with some dendron $D$? 
	
	%1502  writing the following 3 lines in a more careful form 
%	(I am almost sure that it is true but I have no detailed written form. 
%   So I prefer to formulate it more carefully)
%	One may prove this under an additional assumption that the shadow subbase $\mathcal{S}$ of
%   $\tau_s$ is \textit{connected} in the sense of \cite{Mill-Wattel}. The proof is based on 
%   results from \cite{Mill-Wattel}, using superextensions (as in \cite[Theorem 7.3]{Mal-14}). 
	It seems that this is true under an additional assumption that the shadow subbase $\mathcal{S}$ of $\tau_s$ is \textit{connected} in the sense of \cite{Mill-Wattel}. Then it is possible to use results from \cite{Mill-Wattel} and \textit{superextensions} (as in \cite[Theorem 7.3]{Mal-14}). 
\end{remark}

\br

\section{Some consequences of the tameness of actions on dendrites}\label{cons}

%
%\begin{thm}\label{tame}
%For every action of a group $G$ on a dendrite $X$ the dynamical system $(G, X)$ is tame.
%\end{thm}

%
%\begin{proof}
%By \cite{Na} Exercise 10.57, if a sequence of arcs $A_i$ in a dendrite $X$ which converge in
%the Hausdorff metric to the closed set $A$ then $A$ is either an arc or a single point.   
%
%Let  
%\end{proof}

In the following sections we would like to apply Theorem \ref{thm-tame} in order
to strengthen several known results on actions of
groups on dendrites, mainly from \cite{DM-16}, \cite{MN-16} and \cite{Na-12}.
For the following definitions see e.g. \cite{Gl-76}.

\begin{defin}
Let $(G, X)$ be a dynamical system.
\begin{enumerate}
\item
A pair of points $x, y \in X$ is said to be {\em proximal}
if there is a net $g_i \in G$ and a point $z \in X$ with
$\lim g_i x = \lim g_iy =z$.
The system $(G, X)$ is {\em proximal} if every pair of points in $X$ is proximal.
\item
The system $(G, X)$ is {\em strongly proximal}
if for every probability measure $\mu$ on $X$
there is a net $g_i \in G$ and a point $z \in X$ with
$\lim g_i = \del_z$.
\item
An infinite minimal dynamical system $(G, X)$ is said to be {\em extremely proximal} if for 
every nonempty closed subset $A \subsetneq X$ and any nonempty open subset $U \subset X$
there is an element $g \in G$ with $gA \subset U$.
\end{enumerate}
\end{defin}

An action of a group $G$ on a dendron $D$ is called {\em dendro-minimal}
if every $G$-invariant subdendron $C \subset D$ is either all of $D$ or the emptyset.
By Zorn's lemma every group action on a dendron admits a (nonempty) dendro-minimal subdendron.
An arc $C$ in a dendrite $D$ is a {\em free arc} if it contains no end points:
$C \cap End(D) = \emptyset$.

Except for the extreme proximality claim the following results were proven (independently) in
\cite{S-12}, \cite{Mal-13}, \cite{DM-16} and \cite{SY-17}.

\begin{thm}\label{EP}
Suppose $G$ acts on a dendrite $X$ with no finite orbits:
\begin{enumerate}
\item
There is a unique infinite minimal set $M \subset X$.
\item
If the action is dendro-minimal and $X$ has no free arc then $M=X$.
%\item
%If the action is dendro-minimal then it is extremely proximal ???.
%\item
%In particular, if $X=M$ then the action is extremely proximal.
\item
The action on $M$ is extremely proximal.
\item
The action on $X$ is proximal.
\end{enumerate}
\end{thm}

\begin{proof}
(1) \
Claim (1) is proved in \cite[Corollary 4.3]{MN-16}.
In fact, it suffices to assume that there is no fixed point (see remark 4.2 in \cite{DM-16}).

For (2) see Remark 4.7 in \cite{DM-16}.

(3) \ 
By Lemma 4.4 of \cite{DM-16} the set of end points of $X$ is contained (and is dense) in $M$. 
If $x$ is an endpoint then it has a basis for its neighborhoods which consists
of connected open sets $U$ with $|\partial (U)| = 1$ (see Lemma 2.3 in \cite{DM-16}).
Let $K \subsetneq M$ be a closed subset. 
Thus there is a $x \in End(X) \cap K^c$
and a connected open neighborhood $x \in U$ with $|\partial (U)| = 1$.
As $X$ is a dendrite we have that $\overline{U^c}$ is also a dendrite
and, by Lemma 4.3 of \cite{DM-16} there is $g \in G$ with 
$gK \subset g U^c \subset U$.
%Since the endpoints are dense in $M$ this completes the proof.

(4) \
Given $x, y \in X$ there is a sequence $g_j \in G$ such that
the limits $g_j x = x'$ and $\lim g_j y = y'$ exist and are elements of $M$.
Since $x'$ and $y'$ are proximal, it follows that $x$ and $y$ are proximal as well.
\end{proof}

As corollaries we obtain the following: 
%results which were proven in 
%\cite{S-12}, \cite{Mal-13}, \cite{DM-16} and  \cite{SY-17}.

\begin{cor}
Suppose $G$ acts on a dendrite $X$ with no finite orbits and let 
$M \subset X$ be the unique minimal set, then:
\begin{enumerate}
\item
The action of $G$ on $M$ is strongly proximal.
\item
$G$ contains a free group on two generators. In particular $G$ is not amenable.
\item
(\cite[Corollary 8.3]{DM-16})
The amenable radical of $G$ acts trivially on $M$.
\end{enumerate}
\end{cor}

\begin{proof}
It is shown in \cite{Gl-74} that extreme proximality implies strong proximality
and that a group admitting a nontrivial extremely proximal action has a free
subgroup on two generators (see also \cite{Gl-76}).
To prove part (3) consider the action of the amenable radical, say, $R \nor G$ on $M$.
This admits an invariant probability measure, say $\mu$.
By strong proximality there is a sequence $g_i \in G$ and a point $y \in M$
such that $\lim g_i \mu = \del_y$. As $R$ is a normal subgroup it follows that
each translate $g_i\mu$ as well as the limit $\del_y$ are all $R$-invariant measures.
Next we deduce similarly that every point mass $gy, \ g \in G$, is an $R$ fixed point,
and finally conclude that $R$ acts trivially on $M$.  
\end{proof}

\sk

\begin{prop}\label{prop}
Let $(G, X)$ be a metric tame dynamical system. 
\begin{enumerate}
\item
For any element $p \in E(G, X)$ the set $pX$ is an analytic,
hence universally measurable, subset of $X$.
\item
For every $p \in E(G,X)$ and any probability measure $\mu$ on $X$ we have
$\mu(pX)=1$.
\end{enumerate}
\end{prop}

\begin{proof}
(1) 
The system $(G, X)$ being tame,
%2808
by the enveloping semigroup characterization of metric tame systems \cite{GMU}, 
 we have that every element $p \in E(G,X)$ is 
%2808  
Baire class 1, hence 
 Borel measurable. 
It follows that the set $pX$ is an analytic, hence a universally measurable, subset of $X$.

(2)
As $(G,X)$ is tame there is a {\bf sequence} 
%2808 REPLACING g_j by g_n (several places) 
$g_n \in G$ with
$\lim g_n \to p$ in $E(G,X)$.
Given any function $f \in C(X)$ we have  
$\lim f (g_n x) = f(px)$ for every $x \in X$.
We also have that the function $f \circ p$ is Baire class 1,
hence Borel measurable.
By Lebesgue's bounded convergence theorem it follows that
$$
\int f\, d\mu  = \lim \int (f \circ g_n)\, d\mu =
 \int \lim (f \circ g_n)\, d\mu
  = \int (f \circ p)\, d\mu = 1.
 $$
Taking $f = 1_X$ we conclude that $\mu(pX)=1$.
\end{proof}

\sk 

The next theorem strengthens Theorem 10.1 of \cite{DM-16}.

\begin{thm}\label{SP}
Suppose $G$ acts on a dendrite $X$ with no finite orbits,
then the system $(G, X)$ is strongly proximal.
\end{thm}

\begin{proof}
It follows from Theorem \ref{EP} that $E(G, X)$ has a unique minimal ideal $I \subset E(G, X)$
and that $uX$ is a singleton for every minimal idempotent $u \in I$.
And conversely, for every $x \in M$ there is a minimal idempotent $u \in I$
with $uX=\{x\}$.
Applying part (2) of Proposition \ref{prop} 
%0306 
and Theorem \ref{thm-tame} 
 to $u$ and denoting $uX =\{x\}$, we have $\delta_x(uX) =1$,
hence $u_*\mu = \delta_x$.
\end{proof}

\begin{examples} \ 
	\begin{enumerate}
		\item For each one of Wa\.{z}ewski's universal dendrites $X= D_n$, $n=3,4,\dots,\infty$ (\cite{Wa-23}) the dynamical system $(H(X), X)$, where $H(X)$ is the group of homeomorphisms of $X$, is minimal, tame and extremely, hence also strongly, proximal. 
		%1502   I reomoved (???) Kwiatkowska from references because we do not use this presently. Alternatively, we may include her citation here (or somewhere) with a text like this:   
		Recently Kwiatkowska \cite{Kw} gave a description of the universal minimal $G$-system $M(G)$ for the topological group $G=H(X)$. 
		It would be interesting to check if $M(G)$ as a dynamical $G$-system is tame. 
		This space $M(G)$ is not a dendrite (being zero-dimensional) but perhaps 
		it admits some suitable betweenness relation.  
		
			\item Let $T$ be the $\R$-tree built on the Cayley graph of the free group on two generators $F_2$
			with $S = \{a, b, a^{-1}, b^{-1}\}$ as a set of generators.
			Let $X = T \cup Y$ be the natural compactification of $T$ obtained by adding the
			boundary $Y$ comprising the infinite reduced words on the generators $\{a, b, a^{-1}, b^{-1}\}$. 
			Then $X$ is a dendrite
			%1502  why $X$ is a dendrite ? 
			 and the corresponding dynamical system $(F_2, X)$  is tame, with $Y$ as its unique minimal subset,
			and the system $(F_2, Y)$ is extremely proximal.

				\item Let $T$ be the $\R$-tree built on the increasing array of finite groups $\Z/ 2^k\Z$
				and let $Y = \lim_{\leftarrow}\Z/ 2^k\Z$ be the inverse limit of this array
				which can be identified with the dyadic adding machine.
				Let $X = T \cup Y$ be the corresponding compactification of $T$.
				The dynamical system $(\Z, X)$ is a $\Z$-action on a
				%2808  why ?  
				 dendrite, 
				 hence tame, with the (equicontinuous) adding machine $Y$
				as its unique minimal subset.
	\end{enumerate} 
\end{examples}

\br

\section{Tame actions of Amenable groups on dendrites}\label{discrete-spectrum}

Our starting point is the structure theorem for minimal metric tame dynamical systems
of amenable groups \cite{Gl-18} (see also \cite{H}, \cite{KL} and \cite{Gl-str}).

\begin{thm}\label{amen}
Let $\Ga$ be any group and $(G, X)$ a metric tame minimal system that
admits an invariant probability measure $\mu$.
Then 
$X$ is almost automorphic,
i.e. it has the structure $X \overset{\iota}{\to} Z$, where $Z$ is equicontinuous and $\iota$
is an almost one-to-one extension. Moreover, $\mu$ is unique and the map $\iota$ is a measure theoretical isomorphism $\iota : (X, \mu, \Ga) \to (Z, \la, \Ga)$, where  $\la$ is the Haar measure on the 
homogeneous space $Z$.

Thus, when $\Ga$ is amenable, since every $G$-system admits an invariant probability measure, 
the claim above holds for any minimal metric tame $G$-system.
\end{thm}

\vspace{.5cm}

\begin{thm}
Let $(G, X)$ be a tame dynamical system.
Let $E = E(G, X)$ denote its enveloping semigroup
and $I \subset E$ be any minimal left ideal in $E$.
Let $\mathcal{M}$ be the collection of minimal subsets of $X$ and set
$$
\tilde{M} = \bigcup \{M \subset X : 
M \in \mathcal{M}\}.
$$
Let $P_G(X)$ be the space of $G$-invariant probability measures on $X$.
Then
\begin{enumerate}
\item
For any element $p \in E$ the set $pX$ is an analytic,
hence universally measurable, subset of $X$.
\item
$\mu(pX)=1$ for every $\mu \in P_G(X)$ and $p \in E$.
\item
For any element $p \in I$ the set $pX$ is contained in $\tilde{M}$.
\item
If $\mu \in P_G(X)$ is ergodic then there exists $M \in \mathcal{M}$
with $\mu(M)=1$. Thus the minimal subsystem $(M,\mu,G)$ 
is uniquely ergodic and satisfies the conclusions of Theorem \ref{amen}.
\item
For any $\mu \in P_G(X)$
the measure dynamical system $(X,\mu,G)$ has discrete spectrum;
i.e. $L_2(\mu)$ is spanned by the collection of matrix
coefficients of finite dimensional unitary representations of $G$.
\end{enumerate}
\end{thm}

\begin{proof}

For claims (1) and (2) repeat the proofs of the corresponding claims in Proposition \ref{prop}.

Claim (3) is clear.

(4) 
Suppose $\mu \in P_G(X)$ is ergodic. Let $p$ be any 
element of a minimal ideal in $E(G, X)$.
By (2) and (3) $\mu(pX)=1$ and we can find a $\mu$-generic point
$x \in pX$. Let $M =\overline{Gx} \in \mathcal{M}$
\footnote{
For an amenable group $G$ this has the usual meaning. When $G$ is
not amenable we can still consider
a probability measure $m$ on the discrete group $G$
such that $\mu(\{g\}) > 0$ for every $g \in G$, and then find a generic point
for the Markov operator on $L_1(\mu)$ defined by $K f = m * f$.
}. 
Then clearly
$\mu(M) =1$.
As the subsystem $(G, M)$ is minimal and tame, it follows 
from Theorem \ref{amen}
that it is uniquely ergodic and satisfies the conclusion of 
that theorem. 

(5)
Let $\mu \in P_G(X)$ be any invariant measure.
As in the proof of part (2) we see that the map
$V_p : f \mapsto f \circ p$ is a linear operator on $L_2(\mu)$
of norm $\le 1$. 
Since $\ch \circ p = \ch$ it follows that $\|V_p\| =1$.
Moreover, the map $V : p \mapsto V_p$ is
a continuous semigroup homomorphism from $E$ into the 
semigroup of linear contractions of $L_2(\mu)$ equipped with its strong operator topology.
Let $u$ be an idempotent in $I$. Then by (2) $\mu(uX)=1$
and for every $x \in uX$ and every $p \in E$ we have $px = (pu)x$.
It follows that the image of $E$ under $V$ coincides with the image
of $I$, $\{V_p : p \in E\} = \{V_p : p \in I\}$.
Moreover, for every $p \in I$ we have $\mu(pX \cap uX) =1$,
hence $px = (up)x$ $\mu$-a.e. Thus $V_p = V_{up}$ and we 
conclude that $\{V_p : p \in E\} = \{V_{up} : p \in I\}$.
Now $uI$ is a {\bf group} and it follows that this image,
which is also the closure of the Koopman group $\{V_g : g \in G\}$,
is a compact group of unitary operators. 
The Peter-Weyl theorem completes the proof.
\end{proof}

%0306
Now using results of this section and Theorem \ref{thm-tame} we get  

\begin{cor}
Suppose an amenable group $G$ acts on a dendrite $X$, then
\begin{enumerate}
\item
Each minimal subset  $M \subset X$ is as described in Theorem \ref{amen}
(but see Theorem \ref{SY} below).
\item
Thus an ergodic invariant probability measure on $X$ is either 
 a uniform distribution on a finite set, or it is the 
uniqeuly ergodic measure on a minimal infinite almost automorphic $M \subset X$.
\end{enumerate}
\end{cor}

\sk 

\begin{remarks} \label{r:Mobius} \ 
	\ben 
	\item 
	%1502 replacing Mobius by M\"{o}bius  
 In \cite{HWY} the authors show that every tame cascade satisfies the ``M\"{o}bius disjointness conjecture". 
 Their proof is based on the fact that tame cascades have discrete spectrum.
 	\item 
	%Moreover, 
In \cite{AAM} it was shown that every monotone cascade on a local dendrite satisfies the 
%1502 replacing Mobius by M\"{o}bius 
``M\"{o}bius disjointness conjecture". 
\item 
Theorem \ref{t:nullness} shows that (2) can be derived from (1), at least for every 
invertible cascade on a local dendrite.  
\een 
\end{remarks}

%2409

In the first version of our work, posted 26 June,  2018, on the Arxiv, we
posed the following question:

\begin{question}
Is there an amenable group $G$, an action of $G$ on a dendrite $X$,
and a minimal subset $Y \subset X$, such that the system $(G, Y)$ is almost
automorphic but not equicontinuous?
\end{question}

On 4 July, 2018, Shi and Ye provided a negative answer \cite{SY-18}:

\begin{thm}\label{SY}
Let $G$ be an amenable group acting on a dendrite $X$. 
Suppose $K$ is a minimal set in $X$. 
Then 
%1502   $(K,G)$
$(G,K)$ is equicontinuous, and $K$ is either finite or homeomorphic to the Cantor set.
\end{thm}

\br

\section{Cascades on dendrites}\label{casc}

Let $(T, X)$ be 
%2808 WARNING: I think we have a mix of definitions. Sometimes, in this paper, we say cascade meaning not 
%necessarily INVERTIBLE ... 
a cascade; i.e. a $\Z$-dynamical system where $T : X \to X$
is the homeomorphism of $X$ which corresponds to $1 \in \mathbb{Z}$.
%A pair of points $x, y \in X$ is a {\em Li-Yorke pair} if it is a proximal but not asymptotic.
%A cascade is {\em almost distal} if it admits no nontrivial Li-Yorke pairs. 

We recall the following results of Naghmouchi \cite{Na-12}; 
in all of them $f: X \to X$ is a monotone dendrite map. 
The sets $P(f)$, $R(f)$ and $UR(f)$ denote, respectively, the
set of periodic points, recurrent points and uniformly recurrent points of $f$.
The set $\La(f)$ is the union of the $\om$-limit sets.

%
%%Theorem 2.1. ([2, Theorem 3.10]) 
%\begin{thm}
%Let $X$ be compact metric space and let 
%$f : X \to X$ be a continuous map without Li-Yorke pairs. 
%Then $f$ is minimal when it is transitive.
%\end{thm}

%%Theorem 2.2. ([11, Corollary 1.2]) 
%\begin{thm}
%Let $f : X \to X$ be a monotone dendrite map. Then f has no Li-Yorke pairs.
%\end{thm}
%

% Theorem 1.2 
%
\begin{thm}\label{omega}
For any $x \in X$, we have:
 \begin{enumerate}
 \item
 $\om_f(x)$ is a minimal set. 
 \item
$\om_f(x) \subset \overline{P(f)}$.
 \end{enumerate}
\end{thm}
%
%
%
%
%%Corollary 1.3.  
%
%
%
\begin{thm}\label{cantor}
%Let $f : X \to X$ be a monotone dendrite map. 
%Then for 
For any $x \in X$, $\om_f(x)$ is either a finite set or a minimal Cantor set. 
In particular, f is not transitive.
\end{thm}
%
%
%%Theorem 1.4
%
\begin{thm}\label{UR}
%Let $f : X \to X$ be a monotone dendrite map. 
The following equalities hold
$$
UR(f) = R(f) = \La(f) = \overline{P(f)}.
$$
\end{thm}
%
%
%%Corollary 1.5. 
%
\begin{thm}\label{distal}
%Let $f : X \to X$ be a monotone dendrite map. 
%Then 
The restriction map 
$f \rest R(f)$ is a distal homeomorphism.
\end{thm}
%
%
%
%\begin{thm}
%\end{thm}
%

Of course every self homeomorphism $T : X \to X$ of a dendrite $X$ is monotone.

We can now augment some of Naghamouchi's results in \cite{Na-12} as follows
(see also \cite{MN-16} and \cite{MN-18}):

\begin{thm}
Let $T  : X \to X$ be a self homeomorphism of a dendrite $X$
and consider the $\Z$-system $(T, X)$.
\begin{enumerate}
\item
If $M \subset X$ is a minimal subset then it is either finite or an adding machine.
\item
The union of minimal sets 
$$
\tilde{M} = \bigcup \{M \subset X : M \ {\text {is minimal}}\}
$$
is closed.
\item
Every point in $X \setminus \tilde{M}$ is asymptotic to $\tilde{M}$;
i.e. $\om_T(x) \subset \tilde{M}$.
\end{enumerate}
\end{thm}

\begin{proof}
(1)\ 
By Theorem \ref{thm-tame}, the system $(T, X)$ is tame, and so is $M$.
By Theorem \ref{amen} $M$ is almost automorphic. By Theorem \ref{distal} it is distal,
and by Theorem \ref{cantor} it is a Cantor set.
These facts put together imply that $M$ is  a minimal  equicontinuous cascade on a Cantor set;
i.e. an adding machine.

(2)\  Follows from Theorem \ref{UR}.

(3)\  Follows from Theorem \ref{omega}.

\end{proof}

%
%\begin{question}
%Is there an amenable group $G$, an action of $G$ on a dendrite $X$,
%and a minimal subset $Y \subset X$, such that the system $(G, Y)$ is almost
%automorphic but not equicontinuous?
%\end{question}

\vskip 0.7cm

%%%%%%%%%%%%%%%%%%%%%%%%%%%%%%%%%%%%%%%%%%%%%%%%%%%%%%%%%%%%%%%%%
%%%%%%             Bibliography                            %%%%%%
%%%%%%%%%%%%%%%%%%%%%%%%%%%%%%%%%%%%%%%%%%%%%%%%%%%%%%%%%%%%%%%%%

\bibliographystyle{amsplain}

\end{document}